\documentclass[12pt, a4paper]{extarticle}

\usepackage[english]{babel}
\usepackage[utf8]{inputenc}
\usepackage{amsfonts,amsmath}
\usepackage{amsthm, amssymb}
\usepackage{dsfont}
\usepackage{graphicx}
\usepackage{subfig}
\usepackage{wrapfig}
\usepackage{float}

\usepackage{geometry}
\usepackage{ThesisShortcuts}

\newtheorem{theorem}{Theorem}[section]
\newtheorem{lemma}[theorem]{Lemma}
\newtheorem{proposition}[theorem]{Proposition}

\newtheorem{definition}[theorem]{Definition}
\theoremstyle{definition}

\renewcommand{\textbf}[1]{\begingroup\bfseries\mathversion{bold}#1\endgroup}

\newcommand{\indc}{\displaystyle{1\!\!1_{}}}
\newcommand{\Esp}{\mathbb{E}}
\newcommand{\ds}{\displaystyle}

\newcommand{\frc}[2]{\frac{\ds #1}{\ds #2}}
\newcommand{\ens}{\enspace}

\newcommand{\Prb}{\mathbb{P}}

\newcommand{\hLMMD}{\hat{L}}
\newcommand{\hDn}{\hat{L}}
\newcommand{\dd}{\mathrm{d}}
\newcommand{\Reals}{\mathbb{R}}

\newcommand{\prodscal}[2]{< #1 , #2 >} 

\newtheorem{postita}{Post-it}

\title{New normality test in high-dimension with kernel methods}
\author{\textsc{J\'er\'emie Kellner}\\
		\small Laboratoire de Math\'ematiques UMR 8524 CNRS - Universit\'e Lille 1 - MODAL team-project Inria\\
   \small \texttt{jeremie.kellner@ed.univ-lille1.fr}\\
   		\textsc{Alain Celisse}\\
   		\small Laboratoire de Math\'ematiques UMR 8524 CNRS - Universit\'e Lille 1 - MODAL team-project Inria\\
   \small \texttt{celisse@math.univ-lille1.fr}}
\date{}

\begin{document}
\maketitle

\abstract{A new goodness-of-fit test for normality in high-dimension (and Reproducing Kernel Hilbert Space) is proposed.
It shares common ideas with the Maximum Mean Discrepancy (MMD) it outperforms both in terms of computation time and applicability to a wider range of data. 
Theoretical results are derived for the Type-I and Type-II errors. They guarantee the control of Type-I error at prescribed level and an exponentially fast decrease of the Type-II error. 
Synthetic and real data also illustrate the practical improvement allowed by our test compared with other leading approaches in high-dimensional settings.}

\section{Introduction}

Dealing with non-vectorial data such as DNA sequences often requires defining a kernel \cite{Aronszajn}. Further analysis is then carried out in the associated Reproducing Kernel Hilbert Space (RKHS) where data are often assumed to have a Gaussian distribution.
For instance supervised and unsupervised classification are performed in \cite{Bouveyron} by modeling each class as a Gaussian process. 
This key Gaussian assumption is often made implicitly as in Kernel Principal Component Analysis \cite{ZwaldKPCA} to control the reconstruction error \cite{NikolovPCA}, or in \cite{SrivastavaHiDimMeanTest} where a mean equality test is used in high-dimensional setting.
Assessing that crucial assumption appears necessary.

Depending on the (finite or infinite dimensional) structure of the RKHS, Cramer-von Mises-type normality tests \cite{MardiaSkewKurt,HenzeZirkler,SzekelyRizzoENormTest} can be applied. 
However these tests become less powerful as dimension increases (see Table 3 in \cite{SzekelyRizzoENormTest}). 
An alternative approach consists in randomly projecting high-dimensional objects on one-dimensional directions and then applying univariate test on a few randomly chosen marginals \cite{RandomProjTest}. However such approaches also suffer a lack of power (see Section 4.2 in \cite{RandomProjTest}).
More specifically in the RKHS setting, \cite{Gretton_2007} introduced the Maximum Mean Discrepancy (MMD) and design a statistical test to distinguish between the distribution of two samples. However this approach requires \emph{characteristic kernels} \cite{Kernel2sample} and suffers high computational complexity as well as several approximations of the asymptotic distribution.

The main contribution of the present paper is to provide an algorithmically efficient one-sample statistical test of normality for data in a RKHS (of possibly infinite dimension). However the strategy we describe can be easily extended to the two-sample setting.
Section~\ref{high.dimension.tests} introduces goodness-of-fit tests available in high dimensional settings. They will serve as references in our simulation experiments.
The new goodness-of-fit test is described in Section~\ref{sec.new.test}, while its theoretical performance is detailed in Section~\ref{sec.theoretical.assessment} in terms of control of Type-I and Type-II errors.
Finally results of experiments on synthetic and real data highlight the great theoretical and practical improvement allowed by the new statistical test. They are collected in Section~\ref{sec.experimental.results}. 

\section{High-dimensional goodness-of-fit tests}
\label{high.dimension.tests}
\subsection{Statistical test framework}
\sloppy 
Let $(\mathcal{H},\mathcal{A})$ be a measurable space, and $Y_1, \dots, Y_n \in~\mathcal{H}$ denote a sample of \emph{independent and identically distributed} (\textit{i.i.d.}) random variables drawn from an unknown distribution $P\in\mathcal{P}$, where  $\mathcal{P}$ is a set of distributions defined on $\mathcal{A}$. 

Following \cite{LehRom_2005}, let us define the null hypothesis $H_0:\ P\in \mathcal{P}_0$, and the alternative hypothesis $H_1:\ P\not\in \mathcal{P}\setminus \mathcal{P}_0$ for any subset $\mathcal{P}_0$ of $\mathcal{P}$.
The purpose of a statistical test $\mathcal{T}(Y_1,\ldots,Y_n)$ of $H_0$ against $H_1$ is to distinguish between the null ($H_0$) and the alternative ($H_1$) hypotheses.
For instance if $\mathcal{P}_0$ reduces to a univariate Gaussian distribution with mean $\mu_0$ and variance $\sigma_0^2$, $\mathcal{T}(Y_1,\ldots,Y_n)$ determines whether $P=\mathcal{N}(\mu_0,\sigma_0^2)$ is true or not for a prescribed level of confidence $0<\alpha<1$.
\subsection{Projection-based statistical tests}
\label{RndProjTest}
In the high-dimensional setting, several approaches share a common projection idea dating back to the Cramer-Wold theorem extended to infinite dimensional Hilbert space.
\begin{proposition}{(Prop.~2.1 from \cite{RandomProjTest})} 
Let $\mathcal{H}$ be a separable Hilbert space with inner product $<\cdot , \cdot>$, and $Y,Z\in\mathcal{H}$ denote two random variables with respective Borel probability measures $P_Y$ and $P_Z$. 
If for every $h \in \mathcal{H}$
\begin{equation*}
<Y, h> = <Z, h> \text{ in distribution} \ens ,
\end{equation*}
then $ P_Y = P_Z $.
\end{proposition}
Since considering all possible directions $h$ is impossible with high-dimensional $\mathcal{H}$, \cite{RandomProjTest} suggest to randomly choose some of them from a Gaussian measure. Given an \iid sample $Y_1, \dots, Y_n$, a Kolmogorov-Smirnov test is performed from $\prodscal{Y_1}{h} , \ldots, \prodscal{Y_n}{h}$ for each $h$, leading to the test statistic
\begin{equation*}
D_n(h) = \sup_{x \in \Reals} |\hat{F}_n (x) - F_0(x)| \ens ,
\end{equation*}
where $\hat{F}_n (x)$ is the empirical cdf of $(<Y_i, h>)_i$ and $F_0$ denotes the cdf of the $<Z, h>$.

Since \cite{RandomProjTest} proved too few directions lead to a less powerful test, 
this can be repeated for several randomly chosen directions $h$, keeping then the largest value for $D_n(h)$.
However the test statistic is no longer distribution-free (unlike the univariate Kolmogorov-Smirnov one) when the number of directions is larger than 2. Therefore for a given confidence level on the Type-I error, the test threshold (quantile) must be estimated through Monte-Carlo simulations.
\subsection{The Maximum Mean Discrepancy (MMD)}
\label{MMDTest}

Following \cite{Gretton_2007} the gap between two distributions $P$ and $P_0$ can be measured by
\begin{equation}
\label{GeneralGOFStat}
\Delta(P, P_0) = \sup_{f \in \mathcal{F}} |\Esp_{Y \sim P} f(Y) - \Esp_{Z \sim P_0} f(Z)|,
\end{equation}
where $\mathcal{F}$ is a class of real valued functions. 
Such a quantity is called Maximum Mean Discrepancy (MMD).
Regardless of $\mathcal{F}$, \eqref{GeneralGOFStat} only defines a pseudo-metric on probability distributions (see \cite{HilbertEmbed}). 
In particular it is shown $\Delta(\cdot,\cdot)$ becomes a metric if $\mathcal{F}=H(k)$ is the reproducing kernel Hilbert space (RKHS) \cite{HilbertEmbed} associated with a kernel $k=k(\cdot,\cdot)$ that is \emph{characteristic}.
\begin{definition}{(Characteristic kernel)}\\
Let $\mathcal{F} = H(k)$ in \eqref{GeneralGOFStat} for some kernel $k$. 
Then $k$ is a characteristic kernel if $\Delta(P, P_0)=0$ implies $P=P_0$.
\end{definition}

In practice the MMD has to be easily computed although the supremum in \eqref{GeneralGOFStat}. 
One major interest of taking $\mathcal{F}$ as the unit ball of $H(k)$ is that $\Delta(P,P_0)$ can be cast as an easy to compute quantity as follows.
Let us first introduce the Hilbert space embedding of a distribution $P$.
\begin{definition}{(Hilbert space embedding, Lemma 3 from \cite{Kernel2sample})}
Let $P$ be a distribution such that $\Esp_{Y \sim P} \sqrt{k(Y, Y)}<+\infty$.\\
Then there exists $\mu_P \in H(k)$ such that for every $f \in H(k)$,
\begin{equation}
<\mu_P, f> = \Esp f(Y) \ens .
\end{equation}
$\mu_P$ is called the Hilbert space embedding of $P$ in $H(k)$.
\end{definition}
Then $\Delta (P, P_0)$ can be expressed as the gap between the Hilbert space embeddings of $P$ and $P_0$:
\begin{align}
\Delta (P,P_0) & =  \sup_{f \in H(k), ||f|| \leq 1} |\Esp_P f(Y) - \Esp_{P_0} f(Z)|\notag\\
         & =  \sup_{f \in H(k), ||f|| \leq 1} |<\mu_{P} - \mu_{P_0}, f>|\notag\\
         & =  ||\mu_{P} - \mu_{P_0}|| \ens .
         \label{MMDGap}
\end{align}
Since $\mu_P$ can be estimated by $1/n \sum_{i=1}^n k(Y_i,\cdot)$, \cite{Gretton_2007}, an estimator of \eqref{MMDGap} can be derived as
\begin{align*}
  \widehat{\Delta} =  \frac{1}{n} \paren{ \sum_{i,j=1}^n \croch{ k(Y_i,Y_j) + k(Z_i,Z_j) - 2 k(Y_i,Z_j) } }^{1/2} ,
\end{align*}
where $(Y_1,\ldots,Y_n)$ and  $(Z_1,\ldots,Z_n)$ are samples of \iid random variables with respective distributions $P$ and $P_0$.

However the MMD-based approach suffers two main drawbacks: $(i)$ it requires a characteristic kernel, which restricts its applicability, and $(ii)$ the distribution of the test statistic $\widehat{\Delta}$ has to be approximated at two levels, which reduces the statistical test power.
On the one hand, one purpose of the present work is to design a strategy allowing to deal with very general objects. It is typically the setting where no characteristic kernel does necessarily exist, or at least where conditions to check the characteristic property are completely awkward (see \cite{HilbertEmbed}).
On the other hand, the distribution of $\widehat{\Delta}$ is first approximated by its asymptotic one, which is an infinite sum of weighted non-centered chi-squares \cite[Theorem 12]{Kernel2sample}. Second the distribution parameters have to be approximated through the eigendecomposition of a recentered Gram matrix (see Section 3.2 in \cite{FastConsistentKern2Test}), which is computationally costly.

\section{New normality test in RKHS}
\label{sec.new.test}

\subsection{Goal}

Let $X_1, \dots, X_n \in \X$ be \iid random variables. One only require $\X$ can be equipped with a positive definite kernel $k$ associated with $H(k)$. The typical example of $\mathcal{X}$ one may consider is a set of DNA sequences. 

Focusing on $Y_i = k(X_i, .) \in H(k)$ for all $1 \leq i \leq n$, our goal is to test whether $Y_i = k(X_i, .)$ follows a Gaussian distribution $P_0 = \mathcal{N}(\mu, \Sigma)$.
Since $H(k)$ is a function space, a Gaussian variable $Z\in H(k)$ is a \textit{Gaussian process}.
\begin{definition}{(Gaussian process)}\\
\label{def.gaussproc1}
$Z$ is a Gaussian process if there exists a probability space $(\Omega, \mathcal{F}, \mathbb{P})$ such that for any $a_1, \dots, a_n \in \mathbb{R}$ and $x_1, \dots, x_n \in \mathcal{X}$, $\sum_{i=1}^n a_i Y(x_i)$ is a univariate Gaussian random variable on $(\Omega, \mathcal{F}, \mathbb{P})$.\\
Its mean $\mu \in H(k)$ and covariance function $\Sigma: \mathcal{X} \times \mathcal{X} \to \mathbb{R}$ are defined for every $x, y \in \mathcal{X}$ by:
\begin{equation*}
\mu (x) = \mathbb{E} Z(x), \qquad \Sigma (x, y) = \mathrm{cov}(Z(x), Z(y)) \ens .
\end{equation*}
\end{definition}
%
By considering $H(k)$ as a linear space instead of a function space, a Gaussian process can be defined in an equivalent way. 
\begin{definition}{(Gaussian process)}\\
\label{def.gaussproc2}
$Z$ is a Gaussian process if there exists a probability space $(\Omega, \mathcal{F}, \mathbb{P})$ such that for any $f \in H(k)$, $<Z, f>$ is a univariate Gaussian random variable on $(\Omega, \mathcal{F}, \mathbb{P})$.\\
Its mean $\mu \in H(k)$ and covariance operator $\Sigma_{op} \in HS(H(k))$ are defined for every $f, g \in H(k)$ by:
\begin{align*}
<\mu, f> = & \ens \mathbb{E} <Z, f>,\\ 
<\Sigma_{op} f, g> = & \ens \mathrm{cov}(<Z, f>, <Z, g>),
\end{align*}
where $HS(H(k))$ denotes the space of all linear applications $H(k) \to H(k)$ with finite trace (Hilbert-Schmidt operators).
\end{definition}
The means in Definitions~\ref{def.gaussproc1} and \ref{def.gaussproc2} coincide. $\Sigma$ and $\Sigma_{op}$ are linked by the following equality for every $x, y \in \mathcal{X}$
\begin{equation*}
<\Sigma_{op} k(x, .), k(y, .)> = \Sigma(x, y) \ens .
\end{equation*}

Remark that if $\mathcal{X} = \mathbb{R}^d$ and $k = <.,.>_{\mathbb{R}^d}$, then $H(k)$ is the dual of $\mathbb{R}^d$ (that is the set of all linear forms $<x, .>_{\mathbb{R}^d}$ on $\mathbb{R}^d$). In this case, $H(k)$ is isomorphic to $\mathbb{R}^d$ and Gaussian processes in $H(k)$ are reduced to multivariate Gaussian variables in $\mathbb{R}^d$.

\subsection{New test procedure}

Let us assume $\mu$ and $\Sigma$ are known, and also $\Esp Y_i = \mu = 0$ for every $1 \leq i\leq n$, for the sake of simplicity.

\subsubsection{Algorithm}

We provide the main steps of the whole test procedure, which are further detailed in Sections~\ref{DefLMMD}--\ref{sssc.qtl.estim}.
\begin{enumerate}
\item \textbf{Input:} $X_1, \dots, X_n \in \mathcal{X}$, $k: \mathcal{X} \times \mathcal{X} \to \R$ (kernel), $\Sigma$ (covariance function), and $0<\alpha<1$ (test level).

\item Compute $K = \croch{k(X_i, X_j)}_{i, j}$ (Gram matrix) and $C = \croch{\Sigma(X_i, X_j)}_{i,j}$ (covariance matrix).
 
\item Compute $n \hLMMD^2$ (test statistic) from \eqref{gofstat} that depends on $K$ and $C$ (Section \ref{sssec.new.test.stat})

\item 
\begin{enumerate}
  \item Draw $B$ Monte-Carlo samples $X_1^b, \ldots, X_n^b$ under $H_0$, for $b = 1, \ldots, B$.
  \item Compute $\hat{q}_{\alpha, n}$ ($1-\alpha$ quantile of $n \hLMMD^2$ under $H_0$) (Section \ref{sssc.qtl.estim}).
\end{enumerate}

\item \textbf{Output:} Reject $H_0$ if $n \hLMMD^2 > \hat{q}_{\alpha, n}$, and accept otherwise.

\end{enumerate}
The computation time of $\hat{q}_{\alpha, n}$ is of order $\mathcal{O}(Bn^2)$, which is faster than estimating the MMD limit distribution quantile as long as $n\geq B$ (Section~\ref{sssc.qtl.estim} and Section~\ref{ssc.exec.time}).

\subsubsection{Laplace-MMD (L-MMD)}
\label{DefLMMD}

The Laplace-MMD test (L-MMD) follows the same idea as the MMD test, but improves upon it by relaxing the restrictive assumption of \emph{characteristic kernel}.
Using that Laplace transform $\mathcal{L}_U(t) = \Esp_U \exp(tU)$ characterizes the distribution of a random variable $U\in \R$, the gap between two distributions $P$ and $P_0$ can be evaluated by
\begin{align}
\Delta L = & \sup_{f\in H(k), ||f|| \leq 1} \left|\Esp_{Y \sim P} e^{<Y, f>} - \Esp_{Z \sim P_0} e^{<Z, f>}\right| \nonumber \\
= & \sup_{||f|| = 1} \sup_{\abs{t} \leq 1} \left|\mathcal{L}_{<Y, f>}(t) - \mathcal{L}_{<Z, f>}(t)\right| \ens ,
\label{LMMDGap2}
\end{align}
where $<\cdot, \cdot>$ denotes the inner-product in $H(k)$.
Therefore $\Delta L = 0$ implies $<Y, f>$ and $<Z, f>$ have the same distribution for every $f$, which provides $P = P_0$ by the Cramer-Wold theorem (Proposition~\ref{RndProjTest}).

Deriving an efficient test procedure requires to provide a quantity related to \eqref{LMMDGap2} that is easy to compute.
Following \eqref{MMDGap} this is done rephrasing $\exp\paren{<\cdot,\cdot>} = \bar k$ as a new positive definite kernel associated with a new RKHS $H(\bar k)$.
%
%
Thus it allows to get a computable form of \eqref{LMMDGap2}.\\
\begin{theorem}
\label{th1}
Assume $\max\paren{ \E_P e^{||Y||} , \E_{P_0} e^{||Z||} } < \infty$.
Let $\bar{\mu}_P, \bar{\mu}_{P_0} \in H(\bar{k})$ be respective embeddings of $P$ and $P_0$. 
Then,
\begin{equation}
\label{LMMD}
L = L (P, P_0) := ||\bar{\mu}_P - \bar{\mu}_{P_0} ||_{H(\bar{k})} \ens ,
\end{equation}
equals zero if and only if $P = P_0$.
\end{theorem}
\begin{proof}
Introducing $\bar k = \exp\paren{<\cdot,\cdot>}$, it comes
\begin{align*} 
\Delta L \leq & \sup_{h \in H(\bar{k}), ||h|| \leq e^{1/2}} |<\bar{\mu}_P - \bar{\mu}_{P_0} , h>|_{H(\bar{k})} \\
 = & \ e^{1/2} ||\bar{\mu}_P - \bar{\mu}_{P_0} ||_{H(\bar{k})} 
= \ e^{1/2} L (P, P_0) 
\ens ,
\end{align*}
where the inequality results from $\acc{ \bar k\paren{f,\cdot}, ||f||_{H(k)} \leq 1 }$  $\subset$ $\acc{ h \in H(\bar{k}), ||h||_{H(\bar{k})} \leq e^{1/2} }$.
Therefore $L(P, P_0) = 0$ implies $\Delta L=0$ and $P = P_0$. 
Conversely $P=P_0$ implies $\bar{\mu}_P = \bar{\mu}_{P_0}$ and $L(P, P_0) = 0$.
\end{proof}

\subsubsection{New test statistic}
\label{sssec.new.test.stat}

As in \cite{Gretton_2007}, $L$ can be estimated by replacing $\bar{\mu}_P $ with the sample mean $\hat{\bar{\mu}}_P = 1/n \sum_{i=1}^n e^{<Y_i, .>}$, leading to Proposition~\ref{prop.def.L}.
\begin{proposition}\label{prop.def.L}
Assume the null-distribution $P_0$ is Gaussian $\mathcal{N}(0, \Sigma)$ and the largest eigenvalue $\lambda$ of $\Sigma$ is smaller than $1$.
Then the following statistic $n\hLMMD^2$ is an unbiaised estimator of $n L^2$, with
\begin{equation}
n \hLMMD^2  = \frac{1}{n-1} \sum_{i \neq j}^n e^{ k(X_i,X_j) } - 2 \sum_{i = 1}^n e^{ \frac{1}{2} \Sigma(X_i, X_i) }  + n b^2 \ens ,
\label{gofstat}
\end{equation}
where
$ b^2 := ||\bar{\mu}_{P_0}||^2 = \left[\mathrm{det}(I - \Sigma^2)\right]^{-1/2}$.
\end{proposition}
The eigenvalue condition $\lambda < 1$ is not restrictive. With
any $\gamma > 0$ such that $\gamma \lambda < 1$, one can compare $\gamma^{1/2} Y_i$ with $\mathcal{N}(0, \gamma \Sigma)$. 
The Gram matrix becomes $K^{\prime} = \gamma K$ and the covariance matrix $C^{\prime} = \gamma^2 C$.

Since it involves $n \times n$ matrices, the computation time for $n \hLMMD^2$ is the same as that of $\widehat\Delta$ (Section~\ref{MMDTest}), that is of order $\mathcal{O}(n^2)$. 
\begin{proof}
We prove that \eqref{gofstat} is an unbiaised estimator of $n L^2$, that is its mean equals $n L^2$.
\begin{align*}
\Esp n \hLMMD^2  = & \frac{1}{n-1} \sum_{i \neq j}^n \Esp e^{ k(X_i,X_j) } - 2 \sum_{i = 1}^n \Esp e^{ \frac{1}{2} \Sigma(X_i, X_i) }  + n b^2 \\
 = &\frac{1}{n-1} \sum_{i \neq j}^n \Esp e^{<Y_i, Y_j>_{H(k)}} \\
& - 2 \sum_{i = 1}^n \Esp_{Z \sim P_0} e^{ <\Sigma_{op} Y_i, Y_i>_{H(k)} }  + n ||\mu_{P_0}||^2 \\
= & n ||\mu_P ||_{H(\bar{k})}^2 - 2 n <\mu_P, \mu_{P_0}>_{H(\bar{k})} + ||\mu_{P_0} ||_{H(\bar{k})}^2\\
= & n ||\mu_P - \mu_{P_0} ||^2 = n L^2 \ens .
\end{align*}
\end{proof}

\subsubsection{Quantile estimation}
\label{sssc.qtl.estim}
Designing a test with confidence level $0 < \alpha < 1$ requires to compute the smallest $\epsilon>0$ such that $\Prb_{H_0} (n \hLMMD^2 > \epsilon) \leq \alpha$ (Type-I error), which is the $1 - \alpha$ quantile of the $n \hLMMD^2$ distribution under $H_0$ denoted by $q_{\alpha, n}$. Unfortunately $q_{\alpha, n}$ is unknown and has to be estimated.

Our purpose is to improve on the MMD strategy described in \cite{Kernel2sample} in terms of power of detection by considering the \emph{finite sample} null-distribution of $n \hLMMD^2$ rather than the asymptotic one.
The improvement allowed by our strategy is illustrated by empirical results (see Figure~\ref{TypeTwoErrRealD} for instance).

Since the $H_0$-distribution $P_0 = \mathcal{N}(0, \Sigma)$ is known, $B>0$ \iid copies $n \hLMMD^2_{(1)}, \dots, n \hLMMD^2_{(B)}$ of $n \hLMMD^2$ are drawn to estimate $q_{\alpha, n}$.
More precisely for each $1 \leq b \leq B$, 
\begin{align}
n \hLMMD^2_{(b)} = & \ens \frac{1}{n-1} \sum_{i \neq j} e^{<Z^{(b)}_i, Z^{(b)}_j>} + n ||\bar{\mu}_{P_0} ||^2 \notag \\
& - 2 \sum_{i=1}^n \exp\left(\frac{1}{2} <Z^{(b)}_i, \Sigma_{op} Z^{(b)}_i>\right) \ens ,
\end{align}
where $Z_1^{(b)}, \dots, Z_n^{(b)} \stackrel{\iid}{\sim} P_0$.
%
%
Let us consider covariance $\Sigma_{op}$ with a finite eigenvalue decomposition $\Sigma_{op} = \sum_{r = 1}^{d} \lambda_r \Psi_r^{\otimes 2}$, with nonincreasing eigenvalues $\lambda_1\geq \ldots \geq \lambda_d\geq 0$ and eigenvectors $\acc{\Psi_i^{\otimes 2}}_{i=1,\ldots,d}$. 
From $G_{i, r}^{(k)} := \lambda_r^{-1/2} <Z_i^{k}, \Psi_r>$, the $(G_{i, r}^{(k)})_{i, r, k}$s are independent real-valued $\mathcal{N}(0,1)$, which leads to
\begin{align*}
<Z^{(k)}_i, Z^{(k)}_j>  & = \sum_{r = 1}^d \lambda_r G_{i, r}^{(k)} G_{j, r}^{(k)} \\
<Z^{(k)}_i, \Sigma_ {op} Z^{(k)}_i>  & = \sum_{r = 1}^d \lambda_r^2 \left[G_{i, r}^{(k)}\right]^2 \ens .
\end{align*}

Let us now explain how the quantile estimator is computed. 
Assuming these $B$ copies of $n \hLMMD^2$ are ordered in increasing order $n \hLMMD^2_{(1)}\leq  \dots \leq n \hLMMD^2_{(B)}$, let us define 
\begin{align}\label{exp.quantile.estimator}
\quad \hat{q}_{\alpha, n} := n \hLMMD^2_{(\ell)},\qquad  \ell = \lfloor B+2-\alpha (B+1) \rfloor 
\end{align}
where $\lfloor \cdot \rfloor$ denotes the  integer part. 
This particular choice of $\ell$ is completely justified by the Type-I error control provided in Proposition~\ref{PropTypeIerr}. 
Finally the rejection region is defined by
\begin{equation}
\mathcal{R}_{\alpha} = \{ n \hLMMD^2 > \hat{q}_{\alpha, n} \} \ens . 
\end{equation}

Estimating $q_{\alpha, n}$ requires simulating  $B  \times n \times d$ real Gaussian variables $\mathcal{N}(0,1)$ and computing $B$ copies of $n \hLMMD^2$. 
Since with only $n$ observations assuming $d>n$ seems unrealistic, the overall computational complexity is of order $\mathcal{O}(B n^2)$.
Note that the MMD quantile estimation proposed in \cite{Kernel2sample} involves the computation of the eigenvalue decomposition of $n\times n$ matrices, which has a complexity bounded by $\mathcal{O}(n^3 + (n \log^2(n)) \log(b))$, where the precision is of order $2^{-b}$ \cite{MtxEvComplexity}.
Then our strategy is preferable as long as $n$ is large enough with respect to $B$, which is illustrated by Figure~\ref{Exectime}.

\section{Theoretical assessment}
\label{sec.theoretical.assessment}

\subsection{Type-I error}
The estimator of $q_{\alpha}$ defined by \eqref{exp.quantile.estimator} depends on the $\ell$-th ordered statistic $n \hLMMD^2_{(\ell)}$, where $\ell = \lfloor B+2-\alpha (B+1) \rfloor$.
The purpose of the following result is to justify this somewhat unintuitive choice for $\ell$ by considering the Type-I error of the resulting procedure.
\begin{proposition}{(Type-I error)}\\
\label{PropTypeIerr}
Assume $P = P_0$ and $\alpha \geq 1/(B + 1)$. 
With $\hat{q}_{\alpha,n}$ given by \eqref{exp.quantile.estimator}, it comes
\begin{equation}
\alpha - \frac{1}{B + 1} \leq \mathbb{P}(n \hLMMD^2 > \hat{q}_{\alpha, n}) \leq \alpha \ens .
\end{equation}
\end{proposition}
\begin{proof}[Sketch of proof]
The proof is straightforwardly derived from the cumulative function of the order statistic $\hat{q}_{\alpha, n}$, the density of a Beta distribution and the bounds $(1 - \alpha) (B + 1) \leq \l \leq B +2 - \alpha(B + 1)$.
\end{proof}
Note that for a user-specified level $0<\alpha<1$, the L-MMD procedure requires to draw $B\geq 1/\alpha-1$ samples to compute $\hat{q}_{\alpha,n}$.
Besides the upper bound on the Type-I error is tight since the discrepancy between lower and upper bounds is not larger than $1/(B+1)$, which can be made negligible.

\subsection{Type-II error} 
We now assume $P \neq P_0$. Theorem~\ref{PropTypeIIerr} gives the magnitude of the Type-II error, that is the probability of wrongly accepting $H_0$.

Before stating Theorem~\ref{PropTypeIIerr}, let us introduce or recall useful notation.
\begin{itemize}
\item $L = || \bar{\mu}_P - \bar{\mu}_{P_0} ||$
\item $q_{\alpha, n}$ is the $(1 - \alpha)$-quantile of $n \hat{L}^2$ under the null-hypothesis
\item Let $m^{(2)}_P = \Esp_P ||\bar{\phi}(Y) - \bar{\mu}_P ||^2$
\end{itemize}

Since $n \hat{L}^2$ converges weakly to a sum of weighted chi-squares (see \cite{SerflingBook}, p. $194$), $q_{\alpha, n}$ is close to a constant when $n \to +\infty$. $L$ and $m_P^{(2)}$ do not depend on $n$.

The proof for Theorem~\ref{PropTypeIIerr} is provided in Appendix~\ref{proof.prop.t2e}.
\begin{theorem}{(Type II error)}\\
\label{PropTypeIIerr}
Assume $||Y|| \leq M$ ($P$-almost surely) for some $0 < M < +\infty$. \\
Then, for any $n > (q_{\alpha, n} + m_P^{(2)}) L^{-2}$
\begin{align}
\label{TypeIIerrBound}
\mathbb{P}(n \hat{L}^2 \leq \hat{q}_{\alpha, n}) \leq  \exp\left( - \frac{n \left\{L - \sqrt{(q_{\alpha, n} + m_P^{(2)})/(n - 1)} \right\}^2}{f_1(n) + f_2(M, L, n)}    \right) f_3(B, M, L)  \ens ,
\end{align}
where
\begin{align}
f_1(n) & = 2 m_P^{(2)} + \mathcal{O}_n (1 / \sqrt{n}) \notag \\
f_2(M, L, n) & = \left\{\frac{8 \sqrt{2}}{3} L^2 \exp(M^2 / 2) + L \mathcal{O}_n(1 / n)\right\} \left(1+\mathcal{O}_n(1/\sqrt{n})\right) f_1^{1/2}(n) \notag \\
f_3(B, M, L) & = 1 + \frac{3 C_{P_0}}{8 \exp(M^2/2) L^2 \sqrt{2 m_P^{(2)} \alpha B}} + \frac{o_B(1/\sqrt{B})}{\exp(M^2) L^4} \notag \ens ,
\end{align}
where $C_{P_0}$ only depends on $P_0$ and the "$\mathcal{O}_n$" and "$o_B$" terms are idnependent of $L$ and $M$.
\label{MainProposition}
\end{theorem} 
The upper bound in \eqref{TypeIIerrBound} shows an exponential decrease for the Type-II error when $n$ grows. Furthermore, it reflects the expected behaviour of the Type-II error with respect to meaningful quantities 
\begin{itemize}
\item When $L$ decreases, the bound increases which is relevant as the alternative becomes more difficult to detect,
\item When $M$ gets smaller, the departure between $P_0$ and $P$ is widened and as a result the upper bound decreases,
\item When $\alpha$ (Type-I error) decreases, $q_{\alpha, n}$ gets larger and so does the bound.
\end{itemize}

Remark that the assumption $||Y|| \leq M$ $P$-a.s. is fulfilled if a bounded kernel $k$ is considered.

\section{Experiments}
\label{sec.experimental.results}

\subsection{Type-I/II errors study}
\label{ssc.t12e.study}

Empirical performances of L-MMD are inferred on the basis of synthetic data. 
L-MMD is compared with two other procedures: Random Projection (Section~\ref{RndProjTest}) and the asymptotic version of L-MMD denoted "L-MMDa".

We set $\mathcal{X} = \mathbb{R}^d$ and $k = <.,.>_{\mathbb{R}^d}$ (where $d = 25$) and thus $H(k)$ is reduced to $\mathbb{R}^d$. Therefore, L-MMD is used as a multivariate normality test and synthetic data are $d$-dimensional observations drawn from a multivariate Gaussian distribution $\mathcal{N}(\mu, \Sigma)$. 

To control the difficulty level in the experiments, we introduced two parameters $\delta,\lambda\geq 0$ such that
$\mu = \delta \cdot (1, 1/2, \ldots, 1/d)'$, and $ \Sigma = \lambda \cdot \mathrm{diag}(1, 1/4, \dots, 1/d^{2})$, where $\mathrm{diag}(u)$ denotes the diagonal matrix with diagonal equal to $u\in\R^d$. 
For the Random Projection test, data are projected onto a randomly chosen direction generated from a zero-mean Gaussian distribution of covariance $\mathrm{diag}(1, \ldots, d^{-2})$.
\begin{figure*}[!b]
\vskip 0.2in
\begin{center}
\centerline{
\includegraphics[width=\textwidth]{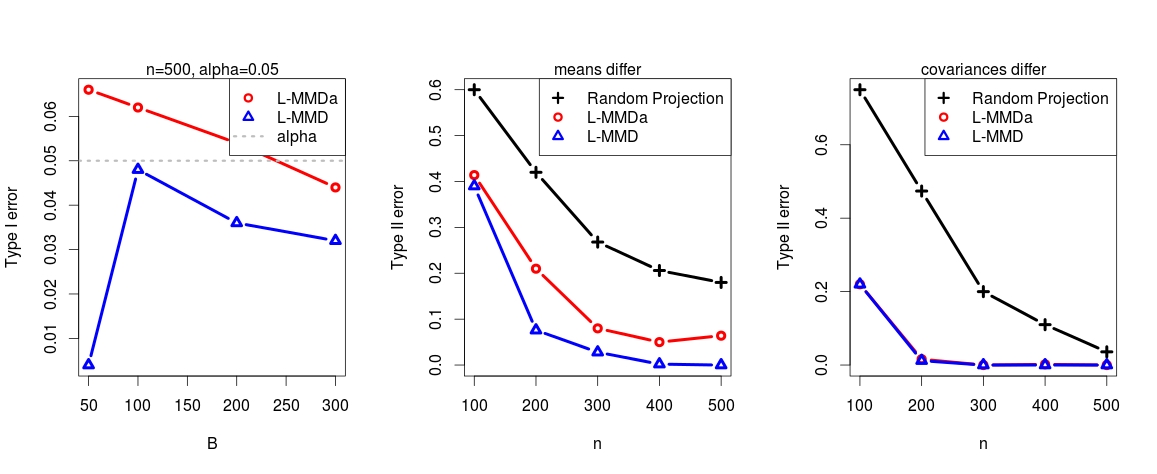}
}
\caption{\textbf{Left:} Type-I errors of the L-MMDa ($\bullet$ red) and L-MMD ($\Delta$ blue) tests. 
\textbf{Center-Right:} Type-II errors of the Random Projection ($+$ black), L-MMDa ($\bullet$ red) and L-MMD ($\Delta$ blue) tests. 
Center: The null-distribution and alternative means differ. 
Right: The null-distribution and the alternative covariances differ.
A theoretical prevision of the L-MMD Type-II error is also plotted (dashed purple).}
\label{T12E_knownparam}
\end{center}
\vskip -0.2in
\end{figure*}

\subsubsection{Type-I error}\label{subsec.typeI.error}

The left panel of Figure \ref{T12E_knownparam} displays the Type-I error of L-MMDa and L-MMD with respect to $B$.
Indeed $B$ independent samples from the asymptotic distribution of $n \hLMMD^2$ have been drawn to allow the comparison between L-MMDa and L-MMD.
Random Projection is not included since it does not depend on $B$ samples. 
Observations are generated from the null-distribution with $\delta_0 = 0$ and $\lambda_0 = 0.5$. 
The test level is $\alpha = 0.05$ and $n = 500$. 
$B$ ranges from $50$ to $300$. 
$500$ simulations are performed for each $B$ and each test.

The Type-I error of L-MMDa is always larger than that of L-MMD although the gap between them remains small ($\leq 0.01$) for $B \geq 100$. 
L-MMD always remains below the prescribed test level $\alpha$ unlike L-MMDa for $B\leq 250$.

\subsubsection{Type-II error}
The same $P_0$ (null-distribution) as in Section~\ref{subsec.typeI.error} is used and two alternatives are considered. 
The first one differs from $P_0$ by the mean ($\delta_{A1} = 0.15$ for the alternative). 
The second one has the same mean as $P_0$ but a different covariance ($\lambda_{A2} = 0.75 \lambda_0$).
Results are displayed in the center (different means) and right (different covariances) of Figure~\ref{T12E_knownparam}. 
We also plotted the prevision of the L-MMD performance provided by Theorem~\ref{PropTypeIIerr}.

L-MMDa and L-MMD both outperform Random Projection that shows the worst overall performance. 
As $n$ grows, L-MMD seems more powerful than L-MMDa ($n\geq 200$).
%

%
\subsection{Influence of the dimensionality}\label{subsec.influence.dimension}
One main concern of goodness-of-fit tests is their drastic loss of power as dimensionality increases. 
Empirical evidences (see Table 3 in \cite{SzekelyRizzoENormTest}) prove ongoing multivariate normality tests suffer such deficiencies.
The purpose of the present section is to check if the good behavior of L-MMD (observed in Section~\ref{ssc.t12e.study} when $d=25$) stills holds in high or infinite dimension.

In Section~\ref{subsubsec.finite.dimension}, two different settings ($d=2$ and $d=25$) are explored with synthetic data where the L-MMD performance is compared with that of two goodness-of-fit tests (Henze-Zirkler and Energy Distance).
Real data serve as infinite dimensional setting in Section~\ref{RealDataPowerStudy} to assess the L-MMD power.

\subsubsection{Finite-dimensional case (Synthetic data)}
\label{subsubsec.finite.dimension}
The power of our test is compared with that of two multivariate normality tests: the HZ test \cite{HenzeZirkler} and the energy distance test \cite{SzekelyRizzoENormTest}. 
In what follows, we briefly recall the main idea of these tests.

The HZ test relies on the following statistic
\begin{equation}
HZ = \int_{\Reals^d} \abs{ \hat{\Psi}(t) - \Psi(t) }^2 \omega(t) dt \ens ,
\end{equation}
where $\Psi(t)$ denotes the characteristic function of $P_0$, $\hat{\Psi}(t) = n^{-1} \sum_{j = 1}^n e^{i <t, Y_j>}$ is the empirical characteristic function of the sample $Y_1, \dots, Y_n$, and $\omega(t) = (2 \pi \beta)^{-d/2} \exp(-||t||^2/(2 \beta))$ with $\beta = 2^{-1/2} [(2d + 1) n)/4]^{1/(d+4)}$. The $H_0$-hypothesis is rejected for large values of $HZ$.
\begin{figure}[!t]
\vskip 0.2in
\begin{center}
\centerline{
\includegraphics[width=\columnwidth]{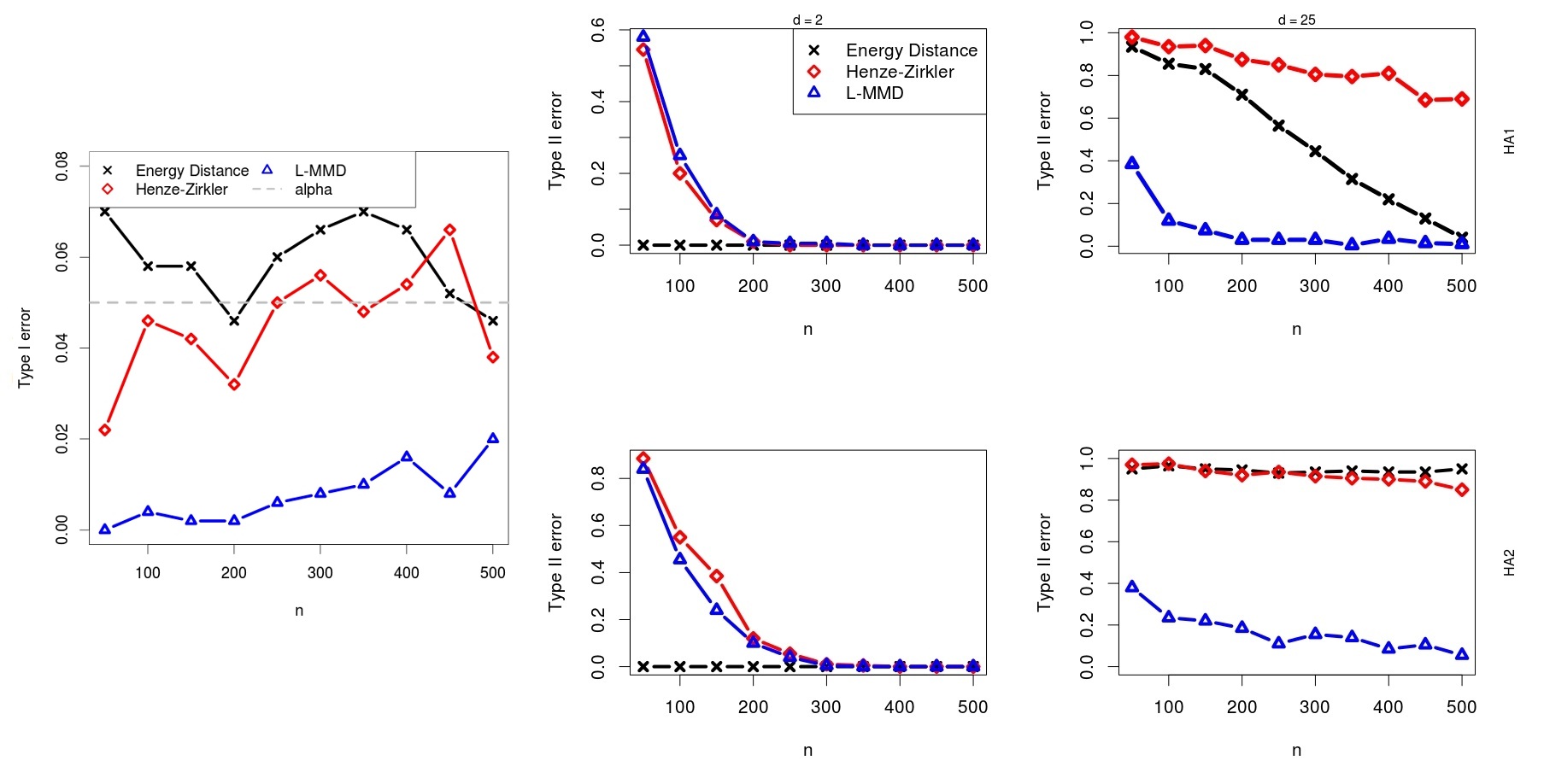}
}
\caption{Type-I and type-II errors of L-MMD ($\Delta$ blue), Energy Distance ($\times$ black), and Henze-Zirkler ($\bullet$ red). For the Type-II error, two alternative distributions are considered: HA1 (top panel) and HA2 (bottom panel). Two settings are considered: $d=2$ (left) and $d = 25$ (right).}
\label{T2E_FinDim}
\end{center}
\vskip -0.2in
\end{figure} 

The energy distance (ED) test is based on 
\begin{equation}
\mathcal{E}(P, P_0) = 2 \Esp ||Y - Z||^2 - \Esp ||Y - Y^{\prime}||^2 - \Esp ||Z - Z^{\prime}||^2 \ens 
\end{equation}
which is called the \emph{energy distance}, where $Y, Y^{\prime} \sim P$ and $Z, Z^{\prime} \sim P_0$. 
Note that $\mathcal{E}(P, P_0) = 0$ if and only if $P = P_0$.
The test statistic is given by
\begin{align}
\hat{\mathcal{E}} = & \frac{2}{n} \sum_{i=1}^n \Esp_Z || Y_i - Z ||^2 - \Esp_{Z, Z^{\prime}} ||Z - Z^{\prime}||^2 \notag \\
         & \hspace*{2cm} -\frac{1}{n^2} \sum_{i, j = 1}^n ||Y_i - Y_j||^2 \hfill \ens ,
\end{align}
where $Z, Z^{\prime} \stackrel{\iid}{\sim} P_0$ (null-distribution).
HZ and ED tests set the $H_0$-distribution at $P_0 = \mathcal{N}(\hat{\mu}, \hat{\Sigma})$ where $\hat{\mu}$ and $\hat{\Sigma}$ are respectively the standard empirical mean and covariance. Therefore, we consider the same null-hypothesis for the L-MMD.

Two alternatives are considered. A mixture of two Gaussians with different means ($\mu_1 = 0$ and $\mu_2 = 1.5 \enspace (1, 1/2, \ldots, 1/d)$) and same covariance $\Sigma = 0.5 \enspace \mathrm{diag}(1, 1/4, \ldots, 1/d^2)$, whose mixture proportions equals either $(0.5, 0.5)$ (alternative HA1) or $(0.8, 0.2)$ (alternative HA2).

$200$ simulations are performed for each test, each alternative and each $n$ (ranging from $100$ to $500$). $B$ is set at $B = 250$ for L-MMD. 

The test level is set at $\alpha = 0.05$ for all tests. Since empirical parameters are considered in all tests, the actual Type-I error may not be controlled anymore. The left plot in Figure \ref{T2E_FinDim} confirms that the actual Type-I error for HZ and ED tests remain more or less around $\alpha$ ($\pm 0.03$). The Type-I error for L-MMD is still upper bounded by $\alpha$ and gets closer to the prescribed test level as $n$ increases.

As for the Type-II error, experimental results (Figure \ref{T2E_FinDim}) reveal two different behaviors as $d$ increases (from center to right columns). 
Whereas both HZ and ED tests lose power, L-MMD still exhibits similar Type-II error values. 
The same conclusion holds true under HA1 and HA2 as well, while the failure of HZ and ED is stronger with HA2 (more difficult). 
This confirms that HZ and ED tests are not suited to high-dimensional settings unlike L-MMD.
\begin{figure}[!b]
\vskip 0.2in
\begin{center}
\centerline{
\includegraphics[width=0.75\columnwidth]{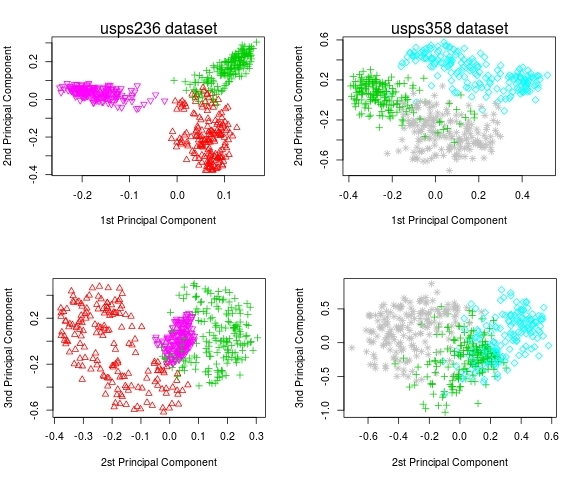}
}
\caption{3D-Visualization (Kernel PCA) of the "Usps236" (left) and "Usps358" (right) datasets}
\label{KPCA_RealD}
\end{center}
\vskip -0.2in
\end{figure} 
Notice that when $d$ is small, L-MMD and HZ have almost the same Type-II error. 
This can be due to the integration involved in the HZ statistic. 
As $d$ increases any discrepancy arising in only a few dimensions is neglected in front of the leading behavior in all other directions.
On the contrary, the supremum at the core of L-MMD \eqref{LMMDGap2} takes into account this kind of discrepancy.

\subsubsection{Infinite-dimensional case (real data)}
\label{RealDataPowerStudy}

Let us consider the USPS dataset (UCI machine learning repository: \textit{http://archive.ics.uci.edu}), which consists of handwritten digits, split up into 10 classes (each for a digit). Each observation represents a $8 \times 8$ greyscale matrix as a $64$-dimensional vector.
A Gaussian kernel $k_G(\cdot,\cdot) = \exp(-\sigma^2 ||\cdot\,-\,\cdot ||^2)$ is used with $\sigma^2 = 10^{-4}$.
Data are visualized through a Kernel PCA \cite{KernelPCA} and displayed in Figure \ref{KPCA_RealD}. 
\begin{figure}[!t]
\vskip 0.2in
\begin{center}
\centerline{
\includegraphics[width=0.75\columnwidth]{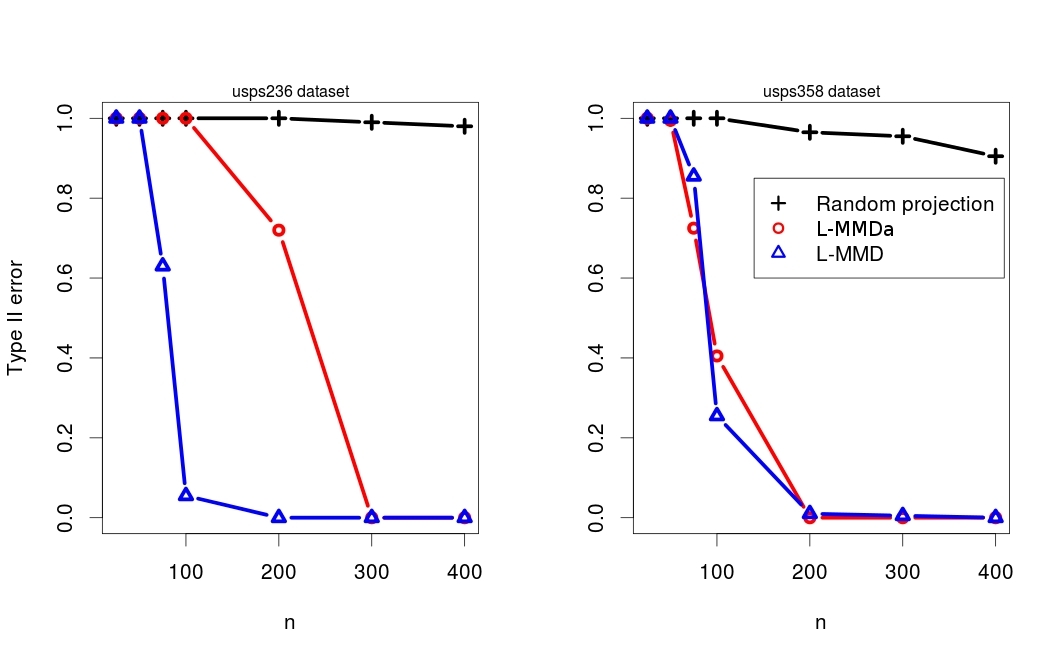}
}
\caption{Comparison of Type-II error for: L-MMD ($\Delta$ blue), L-MMDa ( $\bullet$ red) and Random Projection ( $+$ black). Left: "Usps236". Right: "Usps358".\label{TypeTwoErrRealD}}
\end{center}
\vskip -0.2in
\end{figure} 
Comparing sub-datasets "Usps236" (keeping the three classes "$2$", "$3$" adn "$6$", $541$ observations) and "Usps358" (classes "3", "5" and "8", $539$ observations), the 3D-visualization suggests
three well-separated Gaussian components for ``Usps236'' (left panels), and more overlapping classes for ``Usps358'' (right panels).
Therefore from these two non-Gaussian settings, the last one seems more difficult to detect.

As in Section~\ref{ssc.t12e.study} our test is compared with Random Projection (RP) and L-MMDa tests, specially designed for infinite-dimensional settings. 
For RP, a univariate Kolmogorov-Smirnov test is performed from the projection onto a randomly chosen direction generated by a zero-mean Gaussian process of covariance $k_G$. 
The test level $\alpha = 0.05$ and 100 repetitions have been done for each sample size.  

Results in Figure~\ref{TypeTwoErrRealD} match those obtained in the finite-dimensional case (Figure~\ref{T2E_FinDim}).
On the one hand, RP is by far less powerful than L-MMDa and L-MMD in both cases.
Its Type-II error remains close to 1 whereas L-MMD always truly rejects $H_0$ for $n\geq 200$.
On the other hand, L-MMD seems more powerful than L-MMDa with "Usps236" since it is close to 0 for $n\geq 100$ while L-MMDa reaches similar values only for $n\geq 300$.
However both L-MMDa and L-MMD exhibit a similar behavior in terms of Type-II error with ``Usps358'', and always reject $H_0$ for $n \geq 200$.
This may be due to the higher difficulty of this dataset that do not allow to clearly distinguish between test procedures.
\begin{figure}[!t]
\vskip 0.2in
\begin{center}
\centerline{
\includegraphics[width=0.5 \columnwidth]{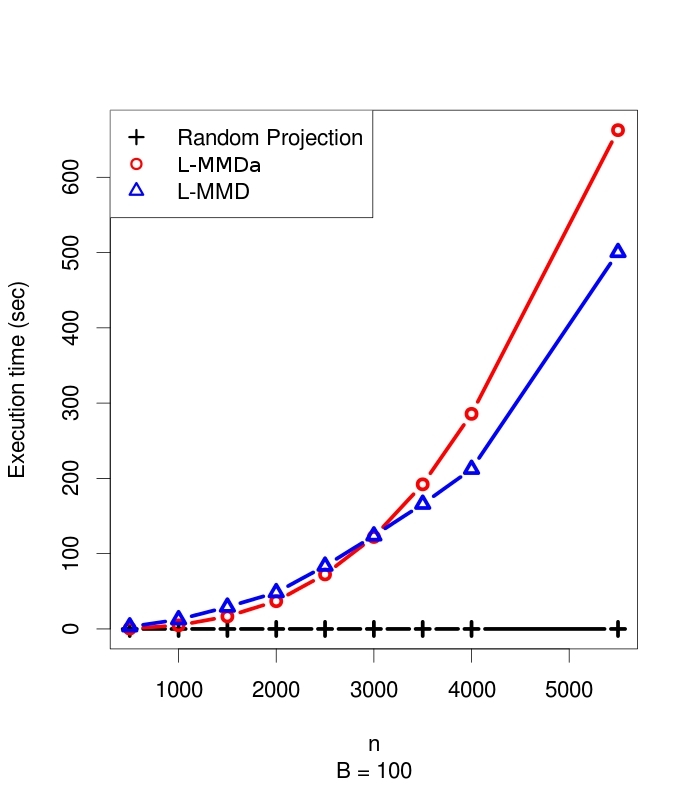}
}
\caption{Execution time of L-MMD ($\Delta$ blue), L-MMDa ( $\bullet$ red) and Random Projection ( $+$ black).\label{Exectime}}
\end{center}
\vskip -0.2in
\end{figure}

\subsection{Execution Time}
\label{ssc.exec.time}
From Sections~\ref{ssc.t12e.study} and~\ref{subsec.influence.dimension} L-MMD is shown to outperform L-MMDa in terms of power. 
This may result from the asymptotic approximation underlying the L-MMDa procedure, while the L-MMD test is performed with the non-asymptotic distribution.
The present section aims at verifying this gain in performance is not balanced by a larger computation time.

From the remark at the end of Section~\ref{sssc.qtl.estim}, L-MMD seems less computationally demanding than L-MMDa as long as $n$ is large enough with respect to $B$.
We carried out an experiment with synthetic data where $B=100$ and $n$ ranges from $500$ to $5500$. No parallelization has been made in this experiment. 
From Figure~\ref{Exectime} results support the above conclusion.  
For $n \leq 3000$, L-MMDa and L-MMD have similar computation time, L-MMDa being only slightly faster. 
However $n>3000$ illustrates the predicted phenomenon. L-MMD is significantly less time consuming than L-MMDa. 
Since the L-MMDa execution time is of order $\mathcal{O}(n^3)$, the L-MMD complexity of order $\mathcal{O}(B n^2)$ becomes smaller as the sample size increases. 
\section{Conclusion}
We introduced a new normality test in RKHS. 
It turns out to be more powerful than ongoing high- or infinite-dimensional tests (such as random projection). 
In particular, empirical studies showed a mild sensibility to high-dimensionality for the L-MMD. Therefore L-MMD can be used as a multivariate normality (MVN) test without suffering a loss of power when $d$ gets larger unlike other MVN tests (Henze-Zirkler, Energy-distance). 

An aspect that most goodness-of-fit tests neglect is the estimation of the distribution parameters (here the mean and covariance of a Gaussian distribution). 
Indeed little is known about how much it affects the test performances. %
Adapting our test to this framework would be welcome in future investigations.

\appendix

\section{Proof of Theorem \ref{PropTypeIIerr}}
\label{proof.prop.t2e}

\subsection{Main proof}

The goal is to get an upper bound for the Type-II error 
\begin{align}
\Prb(n \hDn^2 \leq \hat{q} \mid \mathcal{H}_A) \ens . \label{typeIIerror} 
\end{align}
In the following, the feature map from $H(k)$ to $H(\bar{k})$ will be denoted as
\begin{align*}
\bar{\phi} : H(k) \to H(\bar{k}), \ens y \mapsto \bar{k}(y, .) \ens .
\end{align*}
\begin{enumerate}
\item \textbf{Reduce $n \hDn^2$ to a sum of independent terms}

The first step consists in getting a tight upper bound for \eqref{typeIIerror} which involve a sum of independent terms. This will allow the use of a Bennett concentration inequality in the next step.\\
$n \hDn^2$ is expanded as follows
\begin{align}
n \hDn^2 = & \frac{1}{n-1} \sum_{i \neq j}^n < \bar{\phi}(Y_i) - \bar{\mu}_{P_0},  \bar{\phi}(Y_j) - \bar{\mu}_{P_0}> \notag \\
		:= & n \hDn^2_{P} + n L^2 + 2 n S_n \ens .
		\label{hoeffding.exp}
\end{align}
where $\hat{L}_P^2 = [n(n-1)]^{-1} \sum_{i \neq j}^n < \bar{\phi}(Y_i) - \bar{\mu}_{P},   \bar{\phi}(Y_j) - \bar{\mu}_{P}>$ and $S_n~=~<~\hat{\bar{\mu}}_P - \bar{\mu}_{P}, \bar{\mu}_{P}~-~\bar{\mu}_{P_0}>$ with $\hat{\mu}_P = n^{-1} \sum_{i=1}^n \bar{\phi}(Y_i)$.

It corresponds to the so-called Hoeffding expansion of the U-statistic $\hDn^2$ \cite{HoeffdingUStats} written as a sum of degenerate U-statistics. Since $n \hDn^2_{P}$ converges weakly to a sum of weighted chi-squares and $\sqrt{n} S_n$ to a Gaussian, $\hDn^2_{P}$ becomes negligible with respect to $S_n$ when $n$ is large. Therefore, we consider a surrogate for the Type-II error \eqref{typeIIerror} by removing $\hDn^2_{P}$ with a negligible loss of accuracy.

Using Lemma~\ref{LemmaSplitUStat}, $\hDn^2_{P}$ can be split up into a non-negative quantity and a sum of independent variables
\begin{align}
\hDn^2_{P} = \frac{n}{n-1} ||\hat{\bar{\mu}}_P - \bar{\mu}_P ||^2 - \frac{1}{n(n-1)} \sum_{i=1}^n ||\bar{\phi}(Y_i) - \bar{\mu}_P||^2 \ens .
			  \label{split.L0P2}
\end{align}
Writing \eqref{typeIIerror} conditionally to $\hat{q}$, plugging \eqref{hoeffding.exp} and \eqref{split.L0P2} into \eqref{typeIIerror} and using $||\hat{\bar{\mu}}_P~-~\bar{\mu}_P ||^2 \geq 0$ yield the upper bound

\begin{align}
\Prb(n \hDn^2 \leq \hat{q} \mid \hat{q}) \leq & \ens \Prb\left(- \frac{1}{n-1} \sum_{i=1}^n ||\bar{\phi}(Y_i) - \bar{\mu}_P||^2 + nL^2 + 2n S_n \leq \hat{q} \Vert \hat{q} \right) \label{get.sum.indep} \ens . 
\end{align}
Remark that both positive and negative terms of \eqref{split.L0P2} are of the same order than $\hat{L}^2_P$ (that is of order $n^{-1}$) so that the loss of accuracy in the bound \eqref{get.sum.indep} is negligeable.
\begin{align}
\Prb(n \hDn^2 \leq \hat{q} \mid \hat{q}) \ens \leq \ens  \Prb(\sum_{i=1}^n f(Y_i) \geq n \hat{s} \mid \hat{q}) \label{typeIIerr.first.bound} \ens ,
\end{align}
where 
\begin{align}
f(Y_i) := \ens  \frac{||\bar{\phi}(Y_i) - \bar{\mu}_P||^2}{n-1} - 2<\bar{\phi}(Y_i) - \bar{\mu}_{P}, \bar{\mu}_{P} - \bar{\mu}_{P_0}> \ens , \quad 
\hat{s}	:= \ens L^2 - \frac{\hat{q}}{n} - \frac{m_P^{(2)}}{n-1} \ens , \notag
\end{align}
and $m_P^{(i)} = \Esp ||\bar{\phi(Y_i)} - \bar{\mu}_P||^i$ for any $i \geq 2$.\\
\item \textbf{Apply a concentration inequality}

We now want to find an upper bound for \eqref{typeIIerr.first.bound} through a concentration inequality, namely Lemma~\ref{LemmaBennett} with $\xi_i = f(Y_i)$, $\epsilon = n \hat{s}$, $\nu^2 = \mathrm{Var}(f(Y_i))$ and $f(Y_i) \leq c = \bar{M}$ ($P$-almost surely). 

Lemma~\ref{LemmaBennett} combined with Lemma~\ref{LemmaVariance} and \ref{LemmaBound} yields the upper bound
\begin{align}
\Prb(\sum_{i=1}^n f(Y_i) \geq n \hat{s} \mid  \hat{q}) \leq & \exp\left( - \frac{n \hat{s}^2}{2 \vartheta^2 + (2/3) \overline{M} \vartheta \hat{s}}  \right) \indc_{\hat{s} \geq 0} +  \indc_{\hat{s} < 0} \notag \\
	:= & \exp(g(\hat{s})) \indc_{\hat{s} \geq 0} +  \indc_{\hat{s} < 0} := h(\hat{s}) \label{BennettBound} \ens ,
\end{align}
where
\begin{align}
\overline{M} := \left(4 \sqrt{2} e^{M^2 / 2} L + \frac{m_P^{(2)}}{n - 1}  \right) \ens , \qquad \vartheta^2 := \ens L^2 m_P^{(2)} + \frac{L m^{(3)}_P}{n-1}  + \frac{m_P^{(4)} - (m^{(2)}_P)^2}{4 (n-1)^2} \ens . \notag
\end{align}
\item \textbf{"Replace" the estimator $\hat{q}_{\alpha, n}$ with the true quantile $q_{\alpha, n}$ in the bound}

It remains to take the expectation with respect to $\hat{q}_{\alpha, n}$. In order to make it easy, $\hat{q}_{\alpha, n}$ is pull out of the exponential term of the bound. This is done through a Taylor-Lagrange expansion (Lemma~\ref{LemmaTaylorLag}).

Lemma~\ref{LemmaTaylorLag} rewrites the bound in \eqref{BennettBound} as
\begin{align}
\exp\left( - \frac{n s^2}{2 \vartheta^2 + (2/3) \overline{M} \vartheta  s}  \right) \left\{1 + \frac{3 n }{2 \overline{M} \vartheta} \exp\left( \frac{3 |\tilde{q} - q|}{2 \overline{M} \vartheta }  \right) \indc_{\tilde{s} \geq 0} |\hat{s} - s| \right\} \ens , \label{TaylorLag}
\end{align}
where 
\begin{align}
s = L^2 - \frc{q}{n} - \frc{ b^{(2)}_P}{n-1} \ens , \quad \tilde{s} = L^2 - \frc{\tilde{q}}{n} - \frc{ b^{(2)}_P}{n-1} \ens , \quad \tilde{q} \in (q \wedge \hat{q}, q \vee \hat{q}) \ens \notag,
\end{align}
and $s \geq 0$ because of the assumption $n > (q + m_P^{(2)}) L^{-2}$.

The mean (with respect to $\hat{q}$) of the right-side multiplicative term of \eqref{TaylorLag} is bounded by
\begin{align}
1 + \frac{3 n }{2 \overline{M} \vartheta } \left\{\Esp_{\hat{q}}\left(\exp\left( \frac{3 |\tilde{q} - q|}{\overline{M} \vartheta }  \right) \indc_{\tilde{s} \geq 0} \right)\right\}^{1/2} \sqrt{\Esp_{\hat{q}} (\hat{s} - s)^2} \ens , \notag 
\end{align}

because of the Cauchy-Schwarz inequality.\\
On one hand, using $\hat{q} \to q$ $P_0$-a.s. when $B \to +\infty$
\begin{align}
\Esp_{\hat{q}}\left(\exp\left( \frac{3 |\tilde{q} - q|}{\overline{M} \vartheta }  \right) \indc_{\tilde{s} \geq 0} \right) = & \ens \Esp_{\hat{q}} \left(\left[1 + \frac{o_B(|\hat{q} - q|)}{\overline{M} \vartheta} \right] \indc_{\tilde{s} \geq 0} \right) \notag \\
	\leq & \ens 1 + \frac{\Esp_{\hat{q}} (o_B(|\hat{q} - q|) \indc_{\tilde{s} \geq 0})}{\overline{M} \vartheta} = \ens 1 + \frac{o_B(1)}{\overline{M} \vartheta } \ens ,
	\label{apply.dom.conv.th}
\end{align}
which follows from the Dominated Convergence Theorem (since the variable $|\tilde{q} - q| \indc_{\tilde{s} \geq 0}$ is bounded by the constant $|n L^2 - q| \vee |q|$ for every $B$).\\
On the other hand, Lemma~\ref{LemmaGapQtl} provides
\begin{align}
\Esp (\hat{s} - s)^2 = \frac{\Esp (\hat{q}-q)^2}{n^2} \leq \frac{C_{1, P_0} + \alpha C_{2, P_0} / B }{n^2 \alpha B} \leq \frac{C_{P_0}}{n^2 \alpha B} \ens . 
\end{align}
so that an upper bound for the Type-II error is given by
\begin{align}
& \exp\left( - \frac{n s^2}{2 \vartheta^2 + (2/3) \overline{M} \vartheta s}  \right) \left\{ 1 +  \frac{3 C_{P_0}}{2 \overline{M} \vartheta \sqrt{\alpha B}} +  \frac{o_B(B^{-1/2})}{\overline{M}^2 \vartheta^2}  \right\} \ens . \label{final.upper.bound} 
\end{align}
Finally \eqref{final.upper.bound} can be bounded via the inequalities $n > (q + m_P^{(2)})/L^2$ and $\overline{M} \vartheta \geq 4 \sqrt{2 m_P^{(2)}} \exp(M^2/2) L^2$
\begin{align}
\exp\left( - \frac{n \left[L - \sqrt{(q + m_P^{(2)})/(n - 1)} \right]^2}{f_1(n) + f_2(M, L, n)}    \right) f_3(B, M, L) \ens ,\notag 
\end{align}
where
\begin{align}
f_1(n) & = 2 m_P^{(2)} + \mathcal{O}_n (1 / \sqrt{n}) \notag \\
f_2(M, L, n) & = \left\{\frac{8 \sqrt{2}}{3} L^2 \exp(M^2 / 2) + L \mathcal{O}_n(1 / n)\right\} \left(1+\mathcal{O}_n(1/\sqrt{n})\right) f_1^{1/2}(n) \notag \\
f_3(B, M, L) & = 1 + \frac{3 C_{P_0}}{8 \exp(M^2/2) L^2 \sqrt{2 m_P^{(2)} \alpha B}} + \frac{o_B(1/\sqrt{B})}{\exp(M^2) L^4} \notag \ens .
\end{align}
\end{enumerate}

Theorem~\ref{PropTypeIIerr} is proved.

\subsection{Auxilary results}

\begin{lemma}{(Bennett's inequality, Theorem 2.9 in \cite{Book_ConcIneq})}
\label{LemmaBennett}
Let $\xi_1, \ldots, \xi_n$ i.i.d. zero-mean variables bounded by $c$ and of variance $\nu^2$.\\
Then, for any $\epsilon > 0$
\begin{align}
\Prb\left(\sum_{i=1}^n \xi_i \geq \epsilon \right) \leq \exp\left( - \frac{\epsilon^2}{2 n \nu^2 + 2 c \nu \epsilon / 3}  \right) \ens .
\end{align}
\end{lemma}

\begin{lemma}
\label{LemmaGapQtl}
Assume $\alpha < 1/2$. Then,
\begin{equation}
\Esp (\hat{q}_{\alpha, n} - q_{\alpha, n})^2 \leq \frac{C_{1, P_0}}{\alpha B} + \frac{C_{2, P_0}}{B^2} \ens ,
\end{equation}
where $C_{1,P_0}$ and $C_{2,P_0}$ only depends on $P_0$.
\end{lemma}
\begin{proof}{(Lemma~\ref{LemmaGapQtl})}
Let $U_n=n \hat{L}^2_{\mathcal{H}_0, (\ell)}$ (under the null-hypothesis), $U_{n,1}, \ldots, U_{n, B}$ $B$ i.i.d. copies of $U_n$, $U_{n, (1)} < \ldots < U_{n, (B)}$ the associated order statistics and $q$ the $(1-\alpha)$-quantile of $U_n$, that is $\Prb(U_n > q) = \alpha$.\\
Condider $\ell = \lfloor B + 2 - \alpha (B+1) \rfloor$ and $\hat{q} := U_{n, (\ell)}$. $\Esp (\hat{q} - q)^2$ can be split up the following way
\begin{align}
\Esp (\hat{q} - q)^2 = \mathrm{Var} (\hat{q}) + (\Esp \hat{q} - q )^2 \ens . \label{DecomposeQtlGap}
\end{align}
Theorem 2.9. in \cite{BoucheronOrderStat} provides an upper bound for the variance term when $\ell \geq B/2$ (which holds since $\alpha < 1/2$)
\begin{align}
\mathrm{Var} (\hat{q}) \leq \frac{2}{\alpha B} \Esp h^{-1}(U_{n, (\ell)}) \ens , \notag
\end{align}
where $h$ is the hazard rate of $U_n$ defined by $h = f_n / (1 - F_n)$, $F_n$ is the cumulative distribution function of $U_n$ and $f_n = F_n^{'}$.

Since $U_n$ converges weakly to a (possibly infinite) sum of weighted chi-squares (see \cite{SerflingBook}, p. $194$), $\Esp h^{-1}(U_{n, (\ell)})$ converges to a finite quantity as $n \to +\infty$. Therefore, there exists a quantity $C_{P_0}$ which does not depend on $n$ such that
\begin{align}
\mathrm{Var} (\hat{q}) \leq \frac{C_{P_0}}{\alpha B} \ens .
\label{QtlGap.first.bound}
\end{align}
To bound the second additive term in \eqref{DecomposeQtlGap}, we determine which quantile of $U_n$ $\Esp \hat{q}$ corresponds to.
\begin{align}
\Prb ( U_n \leq \Esp \hat{q}) = & \ens \Esp_{\hat{q}} \Prb(U_n \leq \hat{q} | \hat{q}) = \Esp_{\hat{q}} \Esp_{U_n} \indc_{U_n \leq \hat{q}} \notag \\
	= & \ens \Esp_{U_n} \Esp_{\hat{q}} \indc_{U_n \leq \hat{q}} = \Esp_{U_n} \Prb(\hat{q} \geq U_n | U_n)  \notag \ens .
\end{align}
The expression for the cdf of an order statistic yields
\begin{align}
\Prb ( U_n \leq \Esp \hat{q}) = \Esp_{U_n} \left\{ \sum_{k = B-l}^B \binom{B}{k} (1 - F_n(U_n))^k F_n^{B-k}(U_n) \right\} \notag \ens .
\end{align}
Since $F_n(U_n)$ follows a uniform distribution on $[0, 1]$, the density of a Beta law appears in the latter equation, namely
\begin{align}
\Prb ( U_n \leq \Esp \hat{q}) = & \ens  \sum_{k = B-l}^B \binom{B}{k} \int_0^1 (1 - u)^k u^{B-k} \dd u  \notag \\
	= & \ens \sum_{k = B-l}^B \binom{B}{k} \frac{k! (B-k)!}{(B+1)!} = \frac{\ell + 1}{B+1} \notag \ens .
\end{align}
Let $Q_n = F_n^{-1}$ denote the quantile function of $U_n$.  
Since $U_n$ converges to a sum of weighted chi-squares with quantile function $Q_\infty$, one can write $Q_n = Q_\infty (1 + o_n(1))$ (where $o_n(1)$ holds uniformly on the interval $[1 - \alpha, (1 - \alpha) + 2/(\lceil 2/\alpha \rceil + 1)]$). Besides, let $f_\infty = [Q^{-1}_\infty]^{'}$ the density of the limit distribution of $U_n$.

By the Taylor-Lagrange expansion of $Q_\infty$ of order $1$, hence there exists $\xi \in (1-\alpha, (\ell + 1)/(B +1))$ such that 
\begin{align}
|\Esp \hat{q} - q | = & \ens \left|Q_\infty(\frac{\ell+1}{B+1}) - Q_\infty(1-\alpha)\right| (1 + o_n(1)) \notag \\
	= & \ens \left|Q_\infty(1 - \alpha) + Q_\infty^{'}(\xi)\left( \frac{\ell + 1}{B+1} - (1-\alpha)  \right) - Q_\infty(1-\alpha)\right| (1 + o_n(1)) \notag \\
	\leq & \ens \frac{2 (1 + o_n(1)) }{(B+1) f_\infty(Q_\infty(\xi))} = \frac{2 (1 + o_n(1)) (1 + o_B(1))}{(B+1) f_\infty(q)} \leq \frac{C_{2, P_0}}{B} \ens ,  \label{QtlGap.sec.bound}
\end{align}
where $C_{2, P_0}$ only depends on $P_0$. \\
\eqref{DecomposeQtlGap} combined with \eqref{QtlGap.first.bound} and \eqref{QtlGap.sec.bound} yield the wanted bound.

\begin{lemma}
\label{LemmaSplitUStat}
\begin{align}
\hDn^2_{P} = \frac{n}{n-1} ||\hat{\bar{\mu}}_P - \bar{\mu}_P ||^2 - \frac{1}{n(n-1)} \sum_{i=1}^n ||\bar{\phi}(Y_i) - \bar{\mu}_P||^2 \ens .
\end{align}
\end{lemma}
\begin{proof}
\begin{align}
\hDn^2_{P} = & \frac{1}{n(n-1)} \sum_{i \neq j}^n < \bar{\phi}(Y_i) - \bar{\mu}_{P}, \bar{\phi}(Y_j) - \bar{\mu}_{P}> \notag \\
				= & \frac{1}{n(n-1)} \sum_{i, j = 1}^n < \bar{\phi}(Y_i) - \bar{\mu}_{P}, \bar{\phi}(Y_j) - \bar{\mu}_{P}> - \frac{1}{n(n-1)} \sum_{i = 1}^n || \bar{\phi}(Y_i) - \bar{\mu}_{P} ||^2 \notag \\
				= & \frac{n}{n-1} < \frac{1}{n} \sum_{i = 1}^n \bar{\phi}(Y_i) - \bar{\mu}_{P}, \frac{1}{n} \sum_{j = 1}^n \bar{\phi}(Y_j) - \bar{\mu}_{P}> - \frac{1}{n(n-1)} \sum_{i = 1}^n || \bar{\phi}(Y_i) - \bar{\mu}_{P} ||^2 \notag \\
			  = & \frac{n}{n-1} ||\hat{\bar{\mu}}_P - \bar{\mu}_P ||^2 - \frac{1}{n(n-1)} \sum_{i=1}^n ||\bar{\phi}(Y_i) - \bar{\mu}_P||^2 \ens . \notag
\end{align}
\end{proof}

\begin{lemma}
\label{LemmaBound}
If $||Y|| \leq M$ $P$-a.s., then $f(Y)$ is also bounded
\begin{align}
|f(Y)| \leq \overline{M} := 4 \sqrt{2} e^{M^2 / 2} L + \frac{m_P^{(2)}}{n - 1} \ens .
\end{align}

\end{lemma}

\begin{proof}
\begin{align}
|f(Y)| = & \ens \left| 2 < \bar{\phi}(Y) - \bar{\mu}_P, \bar{\mu}_{P} - \bar{\mu}_{P_0} > - \frac{m_P^{(2)}}{n - 1}  \right| \notag \\
			 \leq & \ens  2 ||\bar{\phi}(Y) - \bar{\mu}_P|| L + \frac{m_P^{(2)}}{n - 1}   \notag \\
			 \leq & \ens  2 \sqrt{2} ||\bar{\phi}(Y)|| L + 2 \sqrt{2} ||\bar{\mu}_P|| L + \frac{m_P^{(2)}}{n - 1}   \notag \\
			 = & \ens  2 \sqrt{2} \Esp e^{||Y||^2 / 2} L + 2 \sqrt{2} \Esp e^{<Y, Y^{'}> / 2} L + \frac{m_P^{(2)}}{n - 1}   \notag \\
			 \leq & \ens 4 \sqrt{2} e^{M^2 / 2} L + \frac{m_P^{(2)}}{n - 1}  := \overline{M}   \ens . \notag
\end{align}
\end{proof}

\begin{lemma}
\label{LemmaVariance}
\begin{align}
\nu^2 \leq \vartheta^2 := L^2 m_P^{(2)} + \frac{L m^{(3)}_P}{n-1}  + \frac{m_P^{(4)} - (m^{(2)}_P)^2}{4 (n-1)^2} \ens .
\end{align}
\end{lemma}

\begin{proof}
\begin{align}
\nu^2 := \mathrm{Var}(g(Y)) = & \ens \Esp <\bar{\phi}(Y_i) - \bar{\mu}_P, \bar{\mu}_P - \bar{\mu}_{P_0} >^2 + \frac{\Esp ||\bar{\phi}(Y_i) - \bar{\mu}_P||^4}{4 (n - 1)^2} \notag \\
	& - \frac{\Esp (<\bar{\phi}(Y_i) - \bar{\mu}_P, \bar{\mu}_P - \bar{\mu}_{P_0}> ||\bar{\phi}(Y_i) - \bar{\mu}_P||^2)}{n - 1} - \left[\frac{\Esp ||\bar{\phi}(Y_i) - \bar{\mu}_P||^2}{2 (n - 1)}\right]^2 \notag \\   
	\leq & \ens L^2 m_P^{(2)} + \frac{L m^{(3)}_P}{n-1}  + \frac{m_P^{(4)} - (m^{(2)}_P)^2}{4 (n-1)^2} := \vartheta^2 \ens , \notag
\end{align}
\end{proof}

\begin{lemma}
\label{LemmaTaylorLag}
Let $h$ be defined as in \eqref{BennettBound}. Then,
\begin{align}
h(\hat{s}) \leq \exp\left( - \frac{n s^2}{2 \vartheta^2 + (2/3) \overline{M} \vartheta  s}  \right) \left\{1 + \frac{3 n }{2 \overline{M} \vartheta} \exp\left( \frac{3 |\tilde{q} - q|}{2 \overline{M} \vartheta }  \right) \indc_{\tilde{s} \geq 0} |\hat{s} - s| \right\} \ens , 
\end{align}
where
\begin{align*}
s = L^2 - \frc{q}{n} - \frc{ b^{(2)}_P}{n-1} \ens , \quad \tilde{s} = L^2 - \frc{\tilde{q}}{n} - \frc{ b^{(2)}_P}{n-1} \ens , \quad \tilde{q} \in (q \wedge \hat{q}, q \vee \hat{q}) \ens .
\end{align*}
\end{lemma}
\begin{proof}
A Taylor-Lagrange expansion of order 1 can be derived for $h(\hat{s})$ since the derivative of $h$ 
\begin{align}
h'(x) = - \frc{(2/3) n \overline{M} \vartheta   x^2 + 4 n \vartheta^2 x}{(2 \vartheta^2 + (2/3) \overline{M} \vartheta  x )^2} \exp\left( - \frc{n x^2}{2 \vartheta^2 + (2/3) \overline{M} \vartheta x}  \right) \indc_{x \geq 0} \notag \ens ,
\end{align}
is well defined for every $x \in \mathbb{R}$ (in particular, the left-side and right-ride derivatives at $x=0$ coincide).\\
Therefore $h(\hat{s})$ equals 
\begin{align}
h(s) + h^{'}(\tilde{s}) & (\hat{s} - s)  \notag \\
& = \exp\left( - \frac{n s^2}{2 \vartheta^2 + (2/3) \overline{M} \vartheta  s}  \right) \left[ 1 + \exp\left( g(s) - g(\tilde{s})  \right) g'(\tilde{s}) \indc_{ \tilde{s} \geq 0} (\hat{s} - s)     \right]  
 \ens ,
\end{align}
where 
\begin{align}
s = L^2 - \frc{q}{n} - \frc{ b^{(2)}_P}{n-1} \ens , \quad \tilde{s} = L^2 - \frc{\tilde{q}}{n} - \frc{ b^{(2)}_P}{n-1} \ens , \quad \tilde{q} \in (q \wedge \hat{q}, q \vee \hat{q}) \ens , \notag
\end{align}
and $s \geq 0$ because of the assumption $n > (q + m_P^{(2)}) L^{-2}$.\\

For every $x, y > 0$, $|g'(x)| \leq 3 n / (2 \overline{M} \vartheta)$ and then $|g(x) - g(y)| \leq 3 n |x - y| / (2 \overline{M} \vartheta )$. It follows
\begin{align}
|g'(\tilde{s})| \leq \frac{3 n }{2 \overline{M} \vartheta } \ens ,
\end{align}
\begin{align}
\exp\left( g(s) - g(\tilde{s})  \right) \leq \exp\left( \frac{3 n}{2 \overline{M} \vartheta } |s - \tilde{s}|  \right) =  \exp\left( \frac{3 |\tilde{q} - q|}{2 \overline{M} \vartheta }  \right)  \ens . \notag
\end{align}

Lemma~\ref{LemmaTaylorLag} is proved.
\end{proof}
\end{proof}
\bibliographystyle{plain}
\bibliography{rkhsgauss-preprint}

\end{document}